\documentclass[11pt]{amsart}
\usepackage{amsmath}
\usepackage{amsthm}
\usepackage{amsfonts,mathrsfs}
\usepackage{amssymb}
\usepackage{graphicx}
\usepackage{color}
\usepackage{verbatim}
\usepackage{wasysym}
\usepackage{enumerate}

\theoremstyle{plain}
\newtheorem{theorem}{Theorem}[section]
\newtheorem{lemma}[theorem]{Lemma}

\newtheorem{corollary}[theorem]{Corollary}

\theoremstyle{definition}
\newtheorem{definition}[theorem]{Definition}

\newtheoremstyle{TheoremNum}
	{\topsep}{\topsep}              
  {\itshape}                      
  {}                              
  {\bfseries}                     
  {.}                             
  { }                             
  {\thmname{#1}\thmnote{ \bfseries #3}}
\newtheorem*{remark}{Remark}
\newcommand{\E}{\mathbb E}
\newcommand{\F}{\mathbb F}
\newcommand{\K}{\mathbb K}

\newcommand{\cN}{\mathcal N}

\newcommand{\cG}{\mathcal G}

\newcommand{\bv}{\mathbf v}

\newcommand{\cC}{\mathcal C}
\newcommand{\rC}{\mathscr C}

\newcommand{\rL}{\mathscr L}

\newcommand{\sN}{\mathrm{N}}

\newcommand{\Aut}{\mathrm{Aut}}

\newcommand{\End}{\mathrm{End}}

\newcommand{\GL}{\mathrm{GL}}

\newcommand{\RN}[1]{%
  \textup{\uppercase\expandafter{\romannumeral#1}}%
}
\newcommand{\rn}[1]{%
  \textup{\lowercase\expandafter{\romannumeral#1}}%
}

\newcommand{\abs}[1]{\lvert#1\rvert}

\renewcommand{\le}{\leqslant}

\def\zhou#1 {\fbox {\footnote {\ }}\ \footnotetext { From Joe: {\color{red}#1}}}

\def\schmidt#1 {\fbox {\footnote {\ }}\ \footnotetext { From Kai: {\color{blue}#1}}}

\date{19 April 2017 (revised 16 September 2017)}
 
\begin{document}
	\title[On the number of inequivalent Gabidulin codes]{On the number of inequivalent\\Gabidulin codes}
	\author[K.-U. Schmidt]{Kai-Uwe Schmidt}
	\address{Department of Mathematics, Paderborn University, 33098 Paderborn, Germany}
	\email{kus@math.upb.de}
	\author[Y. Zhou]{Yue Zhou}
	\address{College of Science, National University of Defense Technology, 410073 Changsha, China}
	\email{yue.zhou.ovgu@gmail.com}
	\maketitle


\begin{abstract}
Maximum rank-distance (MRD) codes are extremal codes in the space of $m\times n$ matrices over a finite field, equipped with the rank metric. Up to generalizations, the classical examples of such codes were constructed in the 1970s and are today known as Gabidulin codes. Motivated by several recent approaches to construct MRD codes that are inequivalent to Gabidulin codes, we study the equivalence issue for Gabidulin codes themselves. This shows in particular that the family of Gabidulin codes already contains a huge subset of MRD codes that are pairwise inequivalent, provided that $2\le m\le n-2$.
\end{abstract}

\section{Introduction}

Let $\K$ be a finite field. The \emph{rank metric} on the $\K$-vector space $\K^{m\times n}$ is defined by
\[
d(A,B)=\mathrm{rk}(A-B) \,\, \text{for} \,\, A,B\in \K^{m\times n}.
\]
We call a subset of $\K^{m\times n}$ equipped with the rank metric a \emph{rank-metric code}. The \emph{minimum distance} of a rank-metric code~$\cC$ is given by
\[
d(\cC)=\min_{A,B\in \cC, A\neq B} d(A,B)
\]
(where we tacitly assume that every rank-metric code contains at least two elements). When $\cC$ is a $\K$-subspace of $\K^{m\times n}$, we say that $\cC$ is a \emph{$\K$-linear} code of dimension $\dim_{\K}(\cC)$. In what follows, we always assume that $m\le n$. It is well known (and easily verified) that every rank-metric code $\cC$ in $\K^{m\times n}$ with minimum distance $d$ satisfies
\[
\abs{\cC}\le \abs{\K}^{n(m-d+1)}.
\]
In case of equality, $\cC$ is called a \emph{maximum} rank-metric code, or \emph{MRD code} for short. MRD codes have been studied since the 1970s and have seen much interest in recent years due to an important application in the construction of error-correcting codes for random linear network coding~\cite{koetter_coding_2008}.
\par
There are several interesting structures in finite geometry, such as quasifields, semifields, and splitting dimensional dual hyperovals, which can be equivalently described as special types of rank-metric codes; see~\cite{dempwolff_dimensional_2014},~\cite{dempwolff_orthogonal_2015},~\cite{johnson_handbook_2007}, \cite{taniguchi_unified_2014}, for example. In particular, a finite quasifield corresponds to an MRD code in $\K^{n\times n}$ with minimum distance~$n$ and a finite semifield corresponds to such an MRD code that is a subgroup of $\K^{n\times n}$ (see~\cite{de_la_cruz_algebraic_2016} for the precise relationship). Many essentially different families of finite quasifields and semifields are known \cite{lavrauw_semifields_2011}, which yield many inequivalent MRD codes in $\K^{n\times n}$ with minimum distance~$n$. In contrast, it appears to be much more difficult to obtain inequivalent MRD codes in $\K^{m\times n}$ with minimum distance strictly less than $m$ (recall that $m\le n$). For the relationship between MRD codes and other geometric objects such as linear sets and Segre varieties, we refer to \cite{lunardon_mrd-codes_2017}.
\par
Based on the classification of the isometries of $\K^{m\times n}$ with respect to the rank metric \cite[Theorem~3.4]{wan_geometry_1996}, we use the following notion of equivalence of rank-metric codes.
\begin{definition}
\label{def:equivalence}
Two rank-metric codes $\cC_1$ and $\cC_2$ in $\K^{m\times n}$ are \emph{equivalent} if there exist $A\in\GL_m(\K)$, $B\in \GL_n(\K)$, $C\in\K^{m\times n}$ and $\rho\in\Aut(\K)$ such that 
\[
\cC_2=\{AX^{\rho}B+C:X \in\cC_1\}
\]
or (but only in the case $m=n$)
\[
\cC_2=\{AX^{\rho}B+C:X^T\in\cC_1\},
\]
where $(\,.\,)^T$ means transposition. 
\end{definition}
\par
Notice that, if $\cC_1$ and $\cC_2$ in Definition~\ref{def:equivalence} are $\K$-linear, then we can without loss of generality let~$C$ be the zero matrix.
\par
A canonical construction of MRD codes was given by Delsarte~\cite{delsarte_bilinear_1978}. This construction was rediscovered by Gabidulin~\cite{gabidulin_MRD_1985} and later generalized by Kshevetskiy and Gabidulin~\cite{kshevetskiy_new_2005}. Today it is customary to call the codes in this generalized family the \emph{Gabidulin codes} (see Section~\ref{sec:codes}, for a precise definition). 
\par
In recent years, an increased interest emerged concerning the question as to whether Gabidulin codes are unique at least for certain parameter sets, or if not, what other constructions can be found. Partial answers were given recently by Horlemann-Trautmann and Marshall~\cite{horlemann-trautmann_new-criteria_2017}, who showed indeed that Gabidulin codes are unique among $\K$-linear MRD codes for certain parameters. On the other hand there are several recent constructions of MRD codes, which were proven to be inequivalent to Gabidulin codes~\cite{cossidente_non-linear_2016},~\cite{csajbok_maximum_2017},~\cite{donati_generalization_2017},~\cite{durante_nonlinear_MRD_2017},~\cite{horlemann-trautmann_new-criteria_2017},~\cite{lunardon_generalized_2015},~\cite{neri_genericity_2017},~\cite{sheekey_new_2016}. 
\par
The aim of this paper is to show that the family of Gabidulin codes in $\K^{m\times n}$ already contains a huge subset of pairwise inequivalent MRD codes, provided that $2\le m\le n-2$. To this end, let $d$ be an integer such that $1\le d\le m\le n$. Gabidulin codes in $\K^{m\times n}$ with minimum distance $d$ can be obtained from Gabidulin codes in $\K^{n\times n}$ with the same minimum distance via projections, obtained by left multiplication with a full-rank $m\times n$ matrix. There are as many as
\[
(q^n-1)(q^n-q)\cdots(q^n-q^{m-1})
\]
projections (where $q=\abs{\K}$) and some of them are obviously equivalent. The main result of this paper is a precise characterization of the equivalence of two projections of a Gabidulin code. This shows that most projections coming from a single Gabidulin code in $\K^{n\times n}$ are pairwise inequivalent, which leads to the following result.
\begin{theorem}
\label{thm:counting_main}
For positive integers $m,n,d$ with $1<d\le m\le n$, there are at least
\[
\frac{1}{n}\prod_{i=2}^m\frac{q^{n-i+1}-1}{q^i-1}
\]
$\K$-linear pairwise inequivalent Gabidulin MRD codes in $\K^{m\times n}$ with minimum distance~$d$.
\end{theorem}
\par
Notice that the lower bound in Theorem~\ref{thm:counting_main} is nontrivial only when $2\le m\le n-2$.
\par
The remainder of this paper is organised as follows. In Section~\ref{sec:pre} we describe rank-metric codes using linearized polynomials, characterize the equivalence between rank-metric codes from this viewpoint, and study nuclei of rank-metric codes. In Section~\ref{sec:codes} we give necessary and sufficient conditions for the equivalence of two projections of a Gabidulin code, from which Theorem~\ref{thm:counting_main} follows.


\section{Rank-metric codes and linearized polynomials}
\label{sec:pre}

We continue using $\K$ to denote a finite field with $q$ elements and let $\F$ be an extension of $\K$ with $[\F:\K]=n$. In this section, we shall describe rank-metric codes in $\K^{m\times n}$ using the language of $\K$-linearized polynomials in $\F[X]$, which are the polynomials in the set
\[
\rL_{\F/\K}=\left\{\sum c_i X^{q^i}: c_i\in \F\right\}.
\]
In what follows, we associate with a given $\K$-subspace $U$ of $\F$ the $\K$-linearized polynomial  
\[
\theta_U=\prod_{u\in U}(X-u)
\]
and let $\bv:\F\to\K^n$ be an isomorphism that maps an element of $\F$ to its coordinate vector with respect to a fixed basis for $\F$ over $\K$.
\begin{lemma}\label{le:polynomials_matrices}
Let $m$ and $n$ be positive integers satisfying $m\le n$. Let $U$ be an $m$-dimensional $\K$-subspace of $\F$ and let $\{\alpha_1,\dots,\alpha_m\}$ be a basis for $U$. Then we have 
\[
\rL_{\F/\K}/(\theta_U)\cong\left\{\left(\bv(f(\alpha_1)),\dots, \bv(f(\alpha_m))\right)^T: f\in\rL_{\F/\K} \right\}.
\]
\end{lemma}
\begin{proof}
The map given by
\begin{center}
\begin{tabular}{cccl}
$\varphi$ :& $\rL_{\F/\K}$ &$\rightarrow$ &$\K^{m\times n}$,\vspace{0.2cm}\\
           & $f$           &$\mapsto$     & $\left(\bv(f(\alpha_1)),\dots,\bv(f(\alpha_m))\right)^T$.
\end{tabular}
\end{center}
is surjective and $\K$-linear. By noting that $\varphi(f)$ is the zero matrix if and only if $f(x)=0$ for every $x\in U$, we see that $\ker(\varphi)=(\theta_U)$, which completes the proof.
\end{proof}
\par
In particular, for $U=\F$, Lemma~\ref{le:polynomials_matrices} implies
\[
\End_{\K}(\F)\cong\rL_{\F/\K}/(X^{q^n}-X),
\]
where $\End_{\K}(\F)$ is the set of endomorphisms on $\F$ as a vector space over $\K$. We shall identify $\End_{\K}(\F)$ with $\rL_{\F/\K}/(X^{q^n}-X)$.
\par
For a $\K$-subspace $U$ of $\F$, we define  
\begin{center}
\begin{tabular}{cccc}
$\pi_U$ :& $\rL_{\F/\K}$ &$\rightarrow$ &$\rL_{\F/\K}/(\theta_U)$,\vspace{0.2cm}\\
         & $f$            &$\mapsto$     & $f+(\theta_U)$.
\end{tabular}
\end{center}
Then we can associate with a subset $\cC$ of $\K^{m\times n}$ an $m$-dimensional subspace~$U$ of $\F$ and identify matrices in $\cC$ with elements of $\rL_{\F/\K}/(\theta_{U})$. In this way, rank-metric codes in $\K^{m\times n}$ can be equivalently investigated using subsets of $\rL_{\F/\K}$.
\par
\begin{lemma}
\label{le:f=0}
Let $U$ be an $m$-dimensional $\K$-subspace of $\F$. Let $\rC$ be a subset of $\rL_{\F/\K}$ and suppose that for all distinct $f,g\in\rC$, the number of solutions $x\in U$ of $f(x)=g(x)$ is strictly smaller than $\abs{U}$. Then $\pi_U$ is injective~on~$\rC$.
\end{lemma}
\begin{proof}
Since $f\equiv g\pmod {\theta_U}$ if and only if $f(x)=g(x)$ for every $x\in U$, the lemma follows.
\end{proof}
\par
\begin{corollary}\label{coro:representation_all}
Let $U$ be an $m$-dimensional $\K$-subspace of $\F$. Let $s$ be an integer such that $\gcd(n,s)=1$. Then the set 
\[
\{a_0X+a_1X^{q^s}+\dots+a_{m-1}X^{q^{s(m-1)}}: a_0,\dots, a_{m-1}\in \F\}
\]
is a complete system of distinct representatives for $\rL_{\F/\K}/(\theta_U)$.
\end{corollary}
\begin{proof}
By \cite[Theorem 5]{gow_galois_2009}, every nonzero polynomial in the above set has at most $q^{m-1}$ zeros and so the result follows from Lemma~\ref{le:f=0}.
\end{proof}
\par
The following lemma characterizes the equivalence between two rank-metric codes using the language of linearized polynomials. It is an immediate consequence of Definition~\ref{def:equivalence}.
\begin{lemma}
\label{le:equivalence}
Let $\rC_1$ and $\rC_2$ be subsets of $\rL_{\F/\K}$, and let $U$ and $W$ be two $m$-dimensional $\K$-subspaces of $\F$ with $m\leq n$. The sets of matrices associated with $\pi_U(\rC_1)$ and $\pi_W(\rC_2)$ are equivalent if and only if there exist $\varphi_1$, $\varphi_2, h \in \rL_{\F/\K}$ and $\rho\in \Aut(\K)$ such that
\begin{enumerate}[(a)]
\setlength{\itemsep}{1ex}
\item $\varphi_1(W)=U$,
\item $\varphi_2(\F)=\F$,
\item $\{\pi_W(\varphi_2\circ f^\rho\circ \varphi_1+h):f\in\rC_1\}=\{\pi_W(g):g\in \rC_2\}$.
\end{enumerate}
(Here $f^\rho=\sum a_i^\rho X^i$ for $f=\sum a_i X^i\in\F[X]$.) If $\pi_W(\rC_1)$ and $\pi_U(\rC_2)$ are both $\K$-linear, then we can always take $h=0$.
\end{lemma}
\par
We also need to introduce the following concept, which is crucially required in determining the automorphism groups of Gabidulin codes in \cite{liebhold_automorphism_2016}. For a subset $\rC$ of $\rL_{\F/\K}$ and a $\K$-subspace $W$ of $\F$, the \emph{right nucleus} of $\pi_W(\rC)$ is defined to be
\[
\cN_r(\pi_W(\rC))=\left\{\varphi\in \End_{\K}(\F):\pi_W(\varphi\circ f)\in  \pi_W(\rC)\;\text{for all $f\in \rC$}\right\}
\] 
and the \emph{middle nucleus} of $\pi_W(\rC)$ is defined to be
\[
\cN_m(\pi_W(\rC))=\left\{\psi\in\End_{\K}(W):\pi_W(f\circ\psi)\in\pi_W(\rC)\;\text{for all $f\in \rC$}\right\}.
\]
Using Lemma~\ref{le:equivalence}, it is readily verified that, if $\pi_W(\rC)$ is $\K$-linear, then both nuclei are invariant under the equivalence of rank-metric codes; see \cite{lunardon_kernels_2017} for details.
\begin{remark}
It appears a bit strange to call $\cN_r(\pi_W(\rC))$ the right nucleus, although $\varphi$ acts via left composition on $\rC$. Indeed the right nucleus is originally defined as a set of matrices, which act via right multiplication on a rank-metric code in $\K^{m\times n}$. The name middle nucleus seems even more unnatural. Originally middle nuclei were only defined for semifields, which correspond to $\K$-linear MRD codes in $\K^{n\times n}$ with minimum distance~$n$. Our definition of the middle nucleus is consistent with that for semifields; see \cite{lunardon_kernels_2017}, in which it is also proved that the middle nucleus of an MRD code is always a field, whereas its right nucleus is not necessarily a field.
\end{remark}
\par
The following lemma relates the nuclei of equivalent MRD codes.
\begin{lemma}\label{le:two_codes_nuclei}
Let $U$ and $W$ be $m$-dimensional $\K$-subspaces of $\F$. Assume that $\pi_U(\rC_1)$ and $\pi_W(\rC_2)$ are $\K$-linear codes equivalent under $(\varphi_2, \varphi_1, \rho)$, where $\varphi_1,\varphi_2 \in \rL_{\F/\K}$ are such that $\varphi_1(W)=U$ and $\varphi_2(\F)=\F$ and $\rho\in\Aut(\K)$.
\begin{enumerate}[(1)]
\item\label{item:le:two_codes_nuclei_1} The map $\tau_m$ defined by
\[
\tau_m:\gamma_1\mapsto\varphi_1^{-1}\circ\gamma_1^{\rho}\circ\varphi_1
\]
is an isomorphism from $\cN_m(\pi_U(\rC_1))$ to $\cN_m(\pi_W(\rC_2))$.

\item\label{item:le:two_codes_nuclei_2} The map $\tau_r: \cN_r(\pi_U(\rC_1))\rightarrow \cN_r(\pi_W(\rC_2))$ defined by
\[
\tau_r:\gamma_2\mapsto\varphi_2\circ\gamma_2^{\rho}\circ\varphi_2^{-1}
\]
is an isomorphism from $\cN_r(\pi_U(\rC_1))$ to $\cN_r(\pi_W(\rC_2))$.
\end{enumerate}
\end{lemma}
\begin{proof} 
By Lemma~\ref{le:equivalence} we have
\[
\{\pi_W(\varphi_2\circ f^\rho\circ \varphi_1) : f\in \rC_1 \}= \{\pi_W(g): g\in \rC_2 \}.
\]
For each $\gamma_2\in \cN_m(\pi_W(\rC_2))$ we have
\[
\pi_U((\varphi_2^{-1}\circ (\varphi_2\circ f^\rho\circ \varphi_1 \circ \gamma_2) \circ \varphi_1^{-1})^{\rho^{-1}})\in \pi_U(\rC_1)
\]
for all $f\in \rC_1$, whence
\[
\pi_U((f^\rho\circ \varphi_1 \circ \gamma_2 \circ \varphi_1^{-1})^{\rho^{-1}})\in \pi_U(\rC_1)
\]
for all $f\in \rC_1$. Thus  
\[
(\varphi_1 \circ \gamma_2 \circ \varphi_1^{-1})^{\rho^{-1}}\in \cN_m(\pi_U(\rC_1)).\]
Let $\gamma_1$ denote $(\varphi_1 \circ \gamma_2 \circ \varphi_1^{-1})^{\rho^{-1}}$. It follows that $\varphi_1^{-1} \circ \gamma_1^{\rho} \circ \varphi_1=\gamma_2$ and so the map $\tau_m$ is an isomorphism from $\cN_m(\pi_U(\rC_1))$ to $\cN_m(\pi_W(\rC_2))$. The properties of $\tau_r$ can be proved similarly. 
\end{proof}


\section{Gabidulin codes}
\label{sec:codes}

We still use $\K$ to denote a finite field with $q$ elements and let $\F$ be an extension of $\K$ with $[\F:\K]=n$.
\par
Let $n,k,s$ be positive integers with $\gcd(s,n)=1$ and $1\le k\le n$. Define
\[
\cG_{k,s}=\{a_0 X+a_1 X^{q^s}+\dots+a_{k-1}X^{q^{s(k-1)}}:a_0,a_1,\dots,a_{k-1}\in \F\},
\]
For $k\le m$, let $U$ be an $m$-dimensional $\K$-subspace of $\F$ with a basis $\{\alpha_1,\dots,\alpha_m\}$. A (projected) \emph{Gabidulin code} is defined as
\[
\left\{\left(\bv(f(\alpha_1)),\dots, \bv(f(\alpha_m))\right)^T:f\in\cG_{k,s}\right\}.
\]
This is an MRD code in $\K^{m\times n}$ with minimum distance $m-k+1$, which is a consequence of the fact that each polynomial in $\cG_{k,s}$ has at most $q^{k-1}$ zeros in $\F$~\cite{gow_galois_2009}~\cite{kshevetskiy_new_2005}. In view of Lemma~\ref{le:polynomials_matrices} we identify this code with $\pi_U(\cG_{k,s})$.
\par
Our main result is the following.
\begin{theorem}
\label{th:iff_mn_codes_all}
Let $k,s,m,n$ be positive integers satisfying $\gcd(n,s)=1$ and $k<m\leqslant n$. Let $U$ and $W$ be two $m$-dimensional $\K$-subspaces of $\F$. Then $\pi_U(\cG_{k,s})$ and $\pi_W(\cG_{k,s})$ are equivalent if and only if $W$ can be mapped to $U$ under the action of
\[
\GL_1(\F)\rtimes \Aut(\F/\K).
\]
\end{theorem}
\par
Before we prove Theorem~\ref{th:iff_mn_codes_all}, we show how Theorem~\ref{thm:counting_main} can be deduced from Theorem~\ref{th:iff_mn_codes_all}. First observe that $\abs{\GL_1(\F)\rtimes \Aut(\F/\K)}=n(q^n-1)$ and that the number of $m$-dimensional $\K$-subspaces of $\F$ equals
\[
{n \brack m}_q=\prod_{i=1}^m\frac{q^{n-i+1}-1}{q^i-1}.
\]
Since every element of $\GL_1(\K)$ fixes all $\K$-subspaces of $\F$, the action of $\GL_1(\F)\rtimes \Aut(\F/\K)$ partitions the set of $m$-dimensional $\K$-subspaces of~$\F$ into at least
\[
\frac{1}{n}{n \brack m}_q\frac{q-1}{q^n-1} 
\]
orbits. Each such orbit gives an MRD code in $\K^{m\times n}$ and these are by Theorem~\ref{th:iff_mn_codes_all} pairwise inequivalent. This establishes Theorem~\ref{thm:counting_main}.
\par
Notice that Theorem~\ref{thm:counting_main} is almost meaningless for $m=n-1$. Indeed, it is readily verified that, for arbitrary $(n-1)$-dimensional $\K$-subspaces~$U$ and $W$ of $\F$, there exists $a\in\F$ such that $W=aU$. This gives the following corollary of Theorem~\ref{th:iff_mn_codes_all}.
\begin{corollary}
\label{coro:equivalence_m=n-1}
Let $k,s,m,n$ be positive integers satisfying $\gcd(n,s)=1$ and $k<m\leqslant n$. Then, for all $(n-1)$-dimensional $\K$-subspaces $U$ of $\F$, the MRD codes $\pi_{U}(\cG_{k,s})$ are equivalent.
\end{corollary}
\par
To prove Theorem~\ref{th:iff_mn_codes_all}, we require the following result that gives the nuclei of projections of Gabidulin codes.
\begin{theorem}
\label{th:middle_right_nuclei}
Let $k,s,m,n$ be positive integers satisfying $k<m\leqslant n$ and $\gcd(s,n)=1$. Let $U$ be an $m$-dimensional $\K$-subspace of $\F$. 
\begin{enumerate}[(1)]
\item\label{item:middle_G} Let $t$ be the largest integer such that $U$ is an $\E$-subspace of $\F$ where $\E$ is an extension of $\K$ with $[\E:\K]=t$. Then the middle nucleus of $\pi_{U}(\cG_{k,s})$ is
\[
\cN_m(\pi_{U}(\cG_{k,s}))=\{cX : c\in \E\}.
\]
\item\label{item:right_G} Let $t$ be the smallest positive integer such that $U$ is contained in an extension $\E$ of $\K$ with $[\E:\K]=t$ and write $r=n/t$. If $1\in U$, then the right nucleus of $\pi_U(\cG_{k,s})$ is
\[
\cN_r( \pi_U(\cG_{k,s}))=\left\{\sum_{i=0}^{r-1} c_i X^{q^{it}}: c_0,\dots,c_{r-1}\in \F\right\}.
\]
\end{enumerate}
\end{theorem}
\par
In the form of matrices, Theorem~\ref{th:middle_right_nuclei} was proved in~\cite{liebhold_automorphism_2016}; for the middle nucleus a proof can also be found in~\cite{morrison_equivalence_2014}. For a proof of Theorem~\ref{th:middle_right_nuclei} in the above form, we refer to \cite{trombetti_nuclei_2016}.
\par
We also require the following lemma.
\begin{lemma}\label{le:mono_to_mono_generalized_middle}
Let $k,s,m,n$ be positive integers satisfying $k<m\le n$ and $\gcd(s,n)=1$. Let $W$ be an $m$-dimensional $\K$-subspace of $\F$ and suppose that there exists $\psi\in \rL_{\F/\K}$ is such that $\pi_{W}(f\circ \psi)\in\pi_{W}(\cG_{k,s}) $ for every $f\in\cG_{k,s}$. Then
\[
\psi(X) \equiv bX \pmod {\theta_W}
\]
for some $b\in\F$.
\end{lemma}
\begin{proof}
Recall that 
\[
\cG_{k,s} = \{a_0X+a_1X^{q^s}+\dots+a_{k-1}X^{q^{s(k-1)}}:a_0,a_1,\dots, a_{k-1}\in\F\}.
\]
By taking $f=X$, we have $\pi_W(\psi(X))\in \pi_W(\cG_{k,s})$.  Hence we can assume that
\begin{equation}
\label{eq:varphi(ax)_1}
\psi(X)\equiv\sum_{i=0}^{k-1}c_i X^{q^{is}} \pmod {\theta_W}
\end{equation}
for some $c_0,c_1,\dots,c_{k-1}\in\F$. We show that $c_1=\dots=c_{k-1}=0$. Assume, for a contradiction, that there exists $i\in\{1,2,\dots,k-1\}$ with $c_i\ne 0$. Let~$j$ be the largest such $i$. Since $0<j<k$, we have $X^{q^{(k-j)s}}\in \cG_{k,s}$. Thus, by taking $f=X^{q^{(k-j)s}}$, we obtain
\[
\pi_W(\psi(X)^{q^{(k-j)s}})\in \pi_W(\cG_{k,s}).
\]
From~\eqref{eq:varphi(ax)_1} we find that
\[
\psi(X)^{q^{(k-j)s}}\equiv \sum_{i=0}^j c_i^{q^{(k-j)s}} X^{q^{(i+k-j)s}} \pmod {\theta_W}.
\]
For $i<j$, the summands belong to $\cG_{k,s}$ and, since $\cG_{k,s}$ is an $\F$-space, we obtain
\[
\pi_W(X^{ks})\in\pi_W(\cG_{k,s}).
\]
Since $1<k<m$, Corollary~\ref{coro:representation_all} gives $\pi_W(X^{ks})\notin\pi_W(\cG_{k,s})$, which leads to the desired contradiction.
\end{proof}
\par
We now prove Theorem~\ref{th:iff_mn_codes_all}.
\begin{proof}[Proof of Theorem~\ref{th:iff_mn_codes_all}]
Assume first that $W$ can be mapped to $U$ under the action of $\GL_1(\F)\rtimes \Aut(\F/\K)$. Then there exist $c\in\F^*$ and $j\in\{0,1,\dots,n-1\}$ such that
\[
U=\{cw^{q^j}:w\in W\}.
\]
Take $\varphi_2 = cX^{q^j}$ and $\varphi_1= X^{q^{n-j}}$. Then, for every
\[
f=\sum_{i=0}^{k-1}a_i X^{q^{is}}\in\cG_{k,s},
\]
we have
\[
\varphi_2 \circ f \circ \varphi_1 =c\left(\sum_{i=0}^{k-1} a_i X^{q^{n-j+is}}\right)^{q^j}=	\sum_{i=0}^{k-1} ca_i^{q^j} X^{q^{is}}
\]
and therefore $\varphi_2 \circ f \circ \varphi_1\in \cG_{k,s}$. One also readily verifies that, for every $g\in\cG_{k,s}$ there exists $f\in\cG_{k,s}$ such that $\varphi_2 \circ f\circ\varphi_1=g$. Lemma~\ref{le:equivalence} then implies that $\pi_U(\cG_{k,s})$ and $\pi_W(\cG_{k,s})$ are equivalent.
\par
Now assume that $\pi_U(\cG_{k,s})$ and $\pi_W(\cG_{k,s})$ are equivalent. It is easy to check that, for each $m$-dimensional $\K$-subspace $V$ of $\F$ and each $x\in\F^*$, the codes $\pi_V(\cG_{k,s})$ and $\pi_{xV}(\cG_{k,s})$ are equivalent. We can therefore assume without loss of generality that $1\in U$ and $1\in W$. Let $t$ be the smallest positive integer such that $U$ is contained in an extension $\E$ of $\K$ with $[\E:\K]=t$. Since $\pi_U(\cG_{k,s})$ and $\pi_W(\cG_{k,s})$ are equivalent, they have the same right nuclei, which we denote by $\cN_r$. Writing $r=n/t$, we then find from Theorem~\ref{th:middle_right_nuclei} that
\begin{equation}
\label{eq:Nuclei_U_W}
\cN_r=\left\{\sum_{i=0}^{r-1} c_i X^{q^{it}}: c_0,\dots,c_{r-1}\in \F\right\}.
\end{equation}
In particular, this implies that $W$ is also contained in $\E$. It follows from~\eqref{eq:Nuclei_U_W} that $\cN_r\cong \E^{r\times r}$ and therefore
\begin{equation}
\label{eq:Normalizer(N_r)}
\sN_{\GL_n(\K)}(\cN_r^\times)\cong\GL_r(\E)\rtimes \Aut(\E/\K),
\end{equation}
where $\sN_{G}(S)$ is the normalizer of $S$ in $G$. The latter identity also appears in~\cite{liebhold_automorphism_2016} and can be proved formally using~\cite[Hilfssatz~3.11,~Chapter 2]{huppert_endliche_1967}, for example.
\par
Now, since $\pi_U(\cG_{k,s})$ and $\pi_W(\cG_{k,s})$ are equivalent, there exist $\varphi_1,\varphi_2\in\rL_{\F/\K}$ and $\rho\in\Aut(\K)$ satisfying the conditions of Lemma~\ref{le:equivalence}, namely $\varphi_1(W) = U$, $\varphi_2(\F)=\F$, and
\begin{equation}
\{\pi_W(\varphi_2\circ f^\rho\circ\varphi_1):f\in\cG_{k,s}\}=\{\pi_W(g):g\in\cG_{k,s}\}.   \label{eqn:G_equivalent}
\end{equation}
Since $f^\rho\in\cG_{k,s}$ for each $f\in\cG_{k,s}$, we can without loss of generality, assume that~$\rho$ is the identity mapping. 
\par
Since the right nuclei of $\pi_U(\cG_{k,s})$ and $\pi_W(\cG_{k,s})$ are both equal to $\cN_r$, we conclude from Lemma~\ref{le:two_codes_nuclei} that $\varphi_2$ belongs to $\sN_{\GL_n(\K)}(\cN_r^\times)$. Since $\GL_r(\E)$ corresponds to the subset of all permutation polynomials in~\eqref{eq:Nuclei_U_W}, we find from~\eqref{eq:Normalizer(N_r)} that
\[
\varphi_2\equiv cX^{q^j}\pmod {X^{q^t}-X}
\]
for some $c\in \F^*$ and some $j\in\{0,1,\dots,n-1\}$. Since $W$ is contained in $\E$, we conclude that $\theta_W$ divides $X^{q^t}-X$ and therefore
\[
\varphi_2\equiv cX^{q^j}\pmod {\theta_W}.
\]
Let
\[
f=\sum_{i=0}^{k-1}a_i X^{q^{is}}\in \cG_{k,s}
\]
and write $\widetilde{\varphi}_1=\varphi_1(X)^{q^j}$. Then we have
\begin{align}
f\circ \widetilde{\varphi}_1&=\sum_{i=0}^{k-1}a_i (\varphi_1(X)^{q^{is}})^{q^j}   \nonumber\\
&=c\left(\sum_{i=0}^{k-1}c^{-q^{n-j}}a_i^{q^{n-j}} \varphi_1(X)^{q^{is}}\right)^{q^j}   \nonumber\\
&\equiv\varphi_2\circ \widetilde{f}\circ\varphi_1\pmod {\theta_W},   \label{eqn:f_phi}
\end{align}
where
\[
\widetilde{f}=\sum_{i=0}^{k-1}c^{-q^{n-j}}a_i^{q^{n-j}}X^{q^{is}}.
\]
Since $\widetilde{f}\in\cG_{k,s}$, we find from~\eqref{eqn:G_equivalent} that $\pi_W(\varphi_2\circ \widetilde{f}\circ\varphi_1)\in\pi_W(\cG_{k,s})$ and therefore, using~\eqref{eqn:f_phi},
\[
\pi_W(f\circ \widetilde{\varphi}_1)\in\pi_W(\cG_{k,s}).
\]
Since $f$ was arbitrary, Lemma~\ref{le:mono_to_mono_generalized_middle} implies that
\[
\widetilde{\varphi}_1(X)\equiv bX\pmod{\theta_W}
\]
for some $b\in\F$. Since $\theta_W(x)=0$ for all $x\in W$, we have
\[
\widetilde{\varphi}_1(W)=bW.
\]
On the other hand, we have
\[
U=\varphi_1(W)=\widetilde{\varphi}_1(W)^{q^{n-j}}
\]
and therefore $U=\{(bw)^{q^{n-j}}:w\in W\}$, as required.
\end{proof}

\section*{Acknowledgment}
Yue Zhou would like to thank the hospitality of the University of Augsburg during his staying as a Fellow of the Alexander von Humboldt Foundation. This work is partially supported by the National Natural Science Foundation of China (No.\ 11401579, 11771451).

\end{document}